\documentclass[reqno]{amsart}
\usepackage{amsmath,amsthm,amssymb,amscd}

\newtheorem{theorem}{Theorem}[section]
\newtheorem{lemma}[theorem]{Lemma}

\theoremstyle{definition}
\newtheorem{definition}[theorem]{Definition}

\theoremstyle{remark}

\newtheorem{conj}[theorem]{Conjecture}
\numberwithin{equation}{section}
\allowdisplaybreaks
\begin{document}

%
%
%
%
%
%
%
%
%

\title[Proof of some conjectural congruences of da Silva and Sellers]
 {Proof of some conjectural congruences of da Silva and Sellers}

\author{Ajit Singh}
\address{Department of Mathematics, Indian Institute of Technology Guwahati, Assam, India, PIN- 781039}
\email{ajit18@iitg.ac.in}

\author{Rupam Barman}
\address{Department of Mathematics, Indian Institute of Technology Guwahati, Assam, India, PIN- 781039}
\email{rupam@iitg.ac.in}

\date{October 27, 2021}




\dedicatory{}

\begin{abstract} 
Let $p_{\{3, 3\}}(n)$ denote the number of $3$-regular partitions in three colours. In a very recent paper, da Silva and Sellers studied certain arithmetic properties of $p_{\{3, 3\}}(n)$. They further conjectured four Ramanujan-like congruences modulo $5$ satisfied by $p_{\{3, 3\}}(n)$. In this article, we confirm the conjectural congruences of da Silva and Sellers using the theory of modular forms.
\end{abstract}

\maketitle
\section{Introduction and statement of result} 
 A partition of a positive integer $n$ is any non-increasing sequence of positive integers, called parts,  whose sum is $n$. Let $\ell$ be a fixed positive integer. An $\ell$-regular partition of a positive integer $n$ is a partition of $n$ such that none of its part is divisible by $\ell$. 
 Let $b_{\ell}(n)$ be the number of $\ell$-regular partitions of $n$. The generating function for $b_{\ell}(n)$ is given by 
\begin{align*}
\sum_{n=0}^{\infty}b_{\ell}(n)q^n=\frac{f_{\ell}}{f_1},
\end{align*}
where $f_k:=(q^k; q^k)_{\infty}=\prod_{j=1}^{\infty}(1-q^{jk})$ and $k$ is a positive integer.
 \par 
 In 2018, Hirschhorn \cite{Hirschhorn} derived a number of congruences for $p_3(n)$ modulo high powers of 3, where $p_3(n)$ denotes the number of partitions of $n$ in three colours. Let $p_{\{3, 3\}}(n)$ denote the number of $3$-regular partitions in three colours, whose generating function is given by
 \begin{align}\label{gen-fun-1}
 \sum_{n=0}^{\infty}p_{\{3,3\}}(n)q^n=\frac{f_3^3}{f_1^3}.
 \end{align}
 In 2019, Gireesh and Mahadeva Naika \cite{Gireesh} studied the function $p_{\{3, 3\}}(n)$, and deduced some congruences modulo powers of 3 for $p_{\{3, 3\}}(n)$. In a very recent paper \cite{silva-sellers}, using elementary generating function manipulations and classical techniques, da Silva and Sellers significantly extend the list of proven arithmetic properties satisfied by $p_{\{3, 3\}}(n)$. They have given parity characterisation for $p_{\{3, 3\}}(2n)$. They provide a complete characterisation for $p_{\{3, 3\}}(n)$ modulo 3. For example, they prove that, for all $n\geq 0$
 \begin{align*}
 p_{\{3, 3\}}(3n+1)&\equiv p_{\{3, 3\}}(3n+2)\equiv 0\pmod 3,\\
 p_{\{3, 3\}}(3n)&\equiv \begin{cases}
 (-1)^{k+\ell} \pmod{3}, & \text{if $n=k(3k-1)/2+\ell(3\ell -1)/2$};\\
 0\pmod{3}, & \text{otherwise}.\nonumber
 \end{cases}
 \end{align*}
 They also find some congruences for $p_{\{3, 3\}}(n)$ modulo $4$ and $9$. They further conjecture four Ramanujan-like congruences modulo 5 satisfied by $p_{\{3, 3\}}(n)$.
\begin{conj}\cite[Conjecture 5.1]{silva-sellers}\label{conj1}
For all $n\geq 0$,
\begin{align}
\label{cong-1} p_{\{3, 3\}}(15n+6)&\equiv 0 \pmod{5},\\
\label{cong-2} p_{\{3, 3\}}(25n+6)&\equiv 0 \pmod{5},\\
\label{cong-3} p_{\{3, 3\}}(25n+16)&\equiv 0 \pmod{5},\\
\label{cong-4} p_{\{3, 3\}}(25n+21)&\equiv 0 \pmod{5}.
\end{align}
\end{conj}
In this article, using the theory of modular forms, we confirm that Conjecture \ref{conj1} is true.
\begin{theorem}\label{thm1}
Conjecture \ref{conj1} is true. 
\end{theorem} 
\section{Proof of Theorem \ref{thm1}}
 We prove Theorem \ref{thm1} using the approach developed in \cite{radu1, radu2}. We define the following matrix groups: 
\begin{align*}
\Gamma & :=\left\{\begin{bmatrix}
a  &  b \\
c  &  d      
\end{bmatrix}: a, b, c, d \in \mathbb{Z}, ad-bc=1
\right\},\\
\Gamma_{\infty} & :=\left\{
\begin{bmatrix}
1  &  n \\
0  &  1      
\end{bmatrix} \in \Gamma : n\in \mathbb{Z}  \right\}.
\end{align*}
For a positive integer $N$, let 
$$\Gamma_{0}(N) :=\left\{
\begin{bmatrix}
a  &  b \\
c  &  d      
\end{bmatrix} \in \Gamma : c\equiv~0\pmod N \right\}.$$
The index of $\Gamma_{0}(N)$ in $\Gamma$ is
\begin{align*}
 [\Gamma : \Gamma_0(N)] = N\prod_{p|N}(1+p^{-1}), 
\end{align*}
where $p$ is a prime divisor of $N$.
\par 
 We now recall some of the definitions and results 
from \cite{radu1, radu2} which will be used to prove our results. For a positive integer $M$, let $R(M)$ be the set of integer sequences $r=(r_\delta)_{\delta|M}$ indexed by the positive divisors of $M$. 
If $r \in R(M)$ and $1=\delta_1<\delta_2< \cdots <\delta_k=M$ 
are the positive divisors of $M$, we write $r=(r_{\delta_1}, \ldots, r_{\delta_k})$. Define $c_r(n)$ by 
\begin{align}
\sum_{n=0}^{\infty}c_r(n)q^n:=\prod_{\delta|M}(q^{\delta};q^{\delta})^{r_{\delta}}_{\infty}=\prod_{\delta|M}\prod_{n=1}^{\infty}(1-q^{n \delta})^{r_{\delta}}.
\end{align}
The approach to proving congruences for $c_r(n)$ developed by Radu \cite{radu1, radu2} reduces the number of coefficients that one must check as compared with the classical method which uses Sturm's bound alone.
\par 
Let $m$ be a positive integer. For any integer $s$, let $[s]_m$ denote the residue class of $s$ in $\mathbb{Z}_m:= \mathbb{Z}/ {m\mathbb{Z}}$. 
Let $\mathbb{Z}_m^{*}$ be the set of all invertible elements in $\mathbb{Z}_m$. Let $\mathbb{S}_m\subseteq\mathbb{Z}_m$  be the set of all squares in $\mathbb{Z}_m^{*}$. For $t\in\{0, 1, \ldots, m-1\}$
and $r \in R(M)$, we define a subset $P_{m,r}(t)\subseteq\{0, 1, \ldots, m-1\}$ by
\begin{align*}
P_{m,r}(t):=\left\{t': \exists [s]_{24m}\in \mathbb{S}_{24m} ~ \text{such} ~ \text{that} ~ t'\equiv ts+\frac{s-1}{24}\sum_{\delta|M}\delta r_\delta \pmod{m} \right\}.
\end{align*}
\begin{definition}
	Suppose $m, M$ and $N$ are positive integers, $r=(r_{\delta})\in R(M)$ and $t\in \{0, 1, \ldots, m-1\}$. Let $k=k(m):=\gcd(m^2-1,24)$ and write  
	\begin{align*}
	\prod_{\delta|M}\delta^{|r_{\delta}|}=2^s\cdot j,
	\end{align*}
	where $s$ and $j$  are nonnegative integers with $j$ odd. The set $\Delta^{*}$ consists of all tuples $(m, M, N, (r_{\delta}), t)$ satisfying these conditions and all of the following.
	\begin{enumerate}
		\item Each prime divisor of $m$ is also a divisor of $N$.
		\item $\delta|M$ implies $\delta|mN$ for every $\delta\geq1$ such that $r_{\delta} \neq 0$.
		\item $kN\sum_{\delta|M}r_{\delta} mN/\delta \equiv 0 \pmod{24}$.
		\item $kN\sum_{\delta|M}r_{\delta} \equiv 0 \pmod{8}$.  
		\item  $\frac{24m}{\gcd{(-24kt-k{\sum_{{\delta}|M}}{\delta r_{\delta}}},24m)}$ divides $N$.
		\item If $2|m$, then either $4|kN$ and $8|sN$ or $2|s$ and $8|(1-j)N$.
	\end{enumerate}
\end{definition}
Let $m, M, N$ be positive integers. For $\gamma=
\begin{bmatrix}
	a  &  b \\
	c  &  d     
\end{bmatrix} \in \Gamma$, $r\in R(M)$ and $r'\in R(N)$, set 
	\begin{align*}
	p_{m,r}(\gamma):=\min_{\lambda\in\{0, 1, \ldots, m-1\}}\frac{1}{24}\sum_{\delta|M}r_{\delta}\frac{\gcd^2(\delta a+ \delta k\lambda c, mc)}{\delta m}
	\end{align*}
and 
	\begin{align*}
	p_{r'}^{*}(\gamma):=\frac{1}{24}\sum_{\delta|N}r'_{\delta}\frac{\gcd^2(\delta, c)}{\delta}.
	\end{align*}
	\begin{lemma}\label{lem1}\cite[Lemma 4.5]{radu1} Let $u$ be a positive integer, $(m, M, N, r=(r_{\delta}), t)\in\Delta^{*}$ and $r'=(r'_{\delta})\in R(N)$. 
	Let $\{\gamma_1,\gamma_2, \ldots, \gamma_n\}\subseteq \Gamma$ be a complete set of representatives of the double cosets of $\Gamma_{0}(N) \backslash \Gamma/ \Gamma_\infty$. 
	Assume that $p_{m,r}(\gamma_i)+p_{r'}^{*}(\gamma_i) \geq 0$ for all $1 \leq i \leq n$. Let $t_{min}=\min_{t' \in P_{m,r}(t)} t'$ and
	\begin{align*}
	\nu:= \frac{1}{24}\left\{ \left( \sum_{\delta|M}r_{\delta}+\sum_{\delta|N}r'_{\delta}\right)[\Gamma:\Gamma_{0}(N)] -\sum_{\delta|N} \delta r'_{\delta}\right\}-\frac{1}{24m}\sum_{\delta|M}\delta r_{\delta} 
	- \frac{ t_{min}}{m}.
 	\end{align*}	
	If the congruence $c_r(mn+t')\equiv0\pmod u$ holds for all $t' \in P_{m,r}(t)$ and $0\leq n\leq \lfloor\nu\rfloor$, then it holds for all $t'\in P_{m,r}(t)$ and $n\geq0$.
	\end{lemma}
	To apply the above lemma, we need the following result which gives us a complete set of representatives of the double coset in  
	$\Gamma_{0}(N) \backslash \Gamma/ \Gamma_\infty$. 
	\begin{lemma}\label{lem2}\cite[Lemma 4.3]{wang} If $N$ or $\frac{1}{2}N$ is a square-free integer, then
		\begin{align*}
		\bigcup_{\delta|N}\Gamma_0(N)\begin{bmatrix}
		1  &  0 \\
		\delta  &  1      
		\end{bmatrix}\Gamma_ {\infty}=\Gamma.
		\end{align*}
	\end{lemma}
\begin{proof}[Proof of Theorem \ref{thm1}]
From \eqref{gen-fun-1}, we have 
\begin{align*}
\sum_{n=0}^{\infty}p_{\{3,3\}}(n)q^n=\frac{(q^3; q^3)_{\infty}^3}{(q; q)_{\infty}^3}.
\end{align*}
We choose $(m,M,N,r,t)=(15,3,15,(-3,3),6)$. We verify that $(m,M,N,r,t) \in \Delta^{*}$ and $P_{m,r}(t)=\{6\}$.
By Lemma \ref{lem2}, we know that $\left\{\begin{bmatrix}
	1  &  0 \\
	\delta  &  1      
	\end{bmatrix}:\delta|15 \right\}$ forms a complete set of double coset representatives of $\Gamma_{0}(N) \backslash \Gamma/ \Gamma_\infty$.
	Let $\gamma_{\delta}=\begin{bmatrix}
	1  &  0 \\
	\delta  &  1      
	\end{bmatrix}$. Let $r'=(30,0,0,0)\in R(15)$ and we use $Sage$ to verify that
	$p_{m,r}(\gamma_{\delta})+p_{r'}^{*}(\gamma_{\delta}) \geq 0$ for each $\delta | N$. Using Lemma \ref{lem1} we have 
\begin{align*}
	\nu&:= \frac{1}{24}\left\{ \left( \sum_{\delta|M}r_{\delta}+\sum_{\delta|N}r'_{\delta}\right)[\Gamma:\Gamma_{0}(N)] -\sum_{\delta|N} \delta r'_{\delta}\right\}-\frac{1}{24m}\sum_{\delta|M}\delta r_{\delta} 
	- \frac{ t_{min}}{m}\\
	&=\frac{1}{24}\left\{ \left( 0+30\right)24 -30\right\}-\frac{1}{24\cdot 15}(-3+9) 
	- \frac{ 6}{15}=\frac{85}{3}.
\end{align*}	
Therefore $\lfloor\nu\rfloor$ is equal to $28$. Using $Sage$ we verify that $p_{\{3,3\}}(15n+6)\equiv0\pmod{5}$ for all $0\leq n\leq 28$. By Lemma \ref{lem1} we conclude that $p_{\{3,3\}}(15n+6)\equiv0\pmod{5}$ for all $n\geq 0$. This completes the proof of \eqref{cong-1}.  
To prove \eqref{cong-2}, we take $(m,M,N,r,t)=(25,3,15,(-3,3),6)$. It is easy to verify that $(m,M,N,r,t) \in \Delta^{*}$ and $P_{m,r}(t)=\{6\}$. We compute that the upper bound in Lemma \ref{lem1} is $\lfloor\nu\rfloor=47$. Following similar steps as shown before, we find that $p_{\{3,3\}}(25n+6)\equiv0\pmod{5}$ for all $n\geq 0$. 
\par 
We now prove \eqref{cong-3} and \eqref{cong-4}. We take $(m,M,N,r,t)=(25,3,15,(-3,3),16)$. It is easy to verify that $(m,M,N,r,t) \in \Delta^{*}$ and $P_{m,r}(t)=\{16,21\}$. Here we also check that $t_{min}=16$. By Lemma \ref{lem2}, we know that $\left\{\begin{bmatrix}
	1  &  0 \\
	\delta  &  1      
	\end{bmatrix}:\delta|15 \right\}$ forms a complete set of double coset representatives of $\Gamma_{0}(N) \backslash \Gamma/ \Gamma_\infty$. Let $\gamma_{\delta}=\begin{bmatrix}
	1  &  0 \\
	\delta  &  1      
	\end{bmatrix}$. Let $r'=(50,0,0,0)\in R(15)$ and we use $Sage$ to verify that
$p_{m,r}(\gamma_{\delta})+p_{r'}^{*}(\gamma_{\delta}) \geq 0$ for each $\delta | N$. We compute that the upper bound in Lemma \ref{lem1} is $\lfloor\nu\rfloor=47$. Using $Sage$ we verify that $p_{\{3,3\}}(25n+t')\equiv0\pmod{5}$ for all $t' \in P_{m,r}(t)$ and for $0\leq n\leq 47$. By Lemma \ref{lem1} we conclude that $p_{\{3,3\}}(25n+t')\equiv0\pmod{5}$ for all $t' \in P_{m,r}(t)$ and for all $n\geq 0$. This completes the proof of the theorem.
\end{proof}



\begin{thebibliography}{999}
\bibitem{silva-sellers}
R. da Silva and J. A. Sellers, {\it Arithmetic properties of 3-regular partitions in three colours}, Bull. Aust. Math. Soc., doi:10.1017/S0004972721000411.

\bibitem{Hirschhorn}
M. D. Hirschhorn, {\it Partitions in 3 colours}, Ramanujan J. 45 (2018), 399--411.

\bibitem{Gireesh}
D. S. Gireesh and M. S. Mahadeva Naika, {\it On 3-regular partitions in 3-colors}, Indian J. Pure Appl.
Math. 50 (2019), 137--148.





\bibitem{radu1}
S. Radu, {\it An algorithmic approach to Ramanujan's congruences}, Ramanujan J.  20 (2)  (2009), 295--302.

\bibitem{radu2}
S. Radu and J. A. Sellers, {\it Congruence properties modulo $5$ and $7$ for the pod  function}, Int. J. Number Theory 7 (8) (2011), 2249--2259.

\bibitem{wang}
L. Wang, {\it Arithmetic properties of $(k, \ell)$- regular bipartitions}, Bull. Aust. Math. Soc.  95 (2017), 353--364.
\end{thebibliography}
\end{document}